\documentclass[a4paper,11pt,fancyhdr]{article}
\usepackage[utf8]{inputenc}
\usepackage{epsfig,amsmath,amssymb,graphicx,fancyheadings,indentfirst,%
  undertilde,color,multirow,slashed,mathtools,listings,cite}
\usepackage[english]{babel}
\usepackage[colorlinks=true,urlcolor=black]{hyperref}
\usepackage{geometry}
\usepackage[makeroom]{cancel}
\usepackage{amsthm}
\usepackage{tikz}
\usepackage{graphicx}
\usepackage{textcomp}
\usepackage{tikz-cd}
\usepackage{appendix}
\usepackage{subcaption}
\usetikzlibrary{decorations.markings}
\usepackage[ruled,vlined]{algorithm2e}
\usepackage{mathrsfs}
\usepackage{authblk}
\usetikzlibrary{arrows}

\definecolor{lightgrey}{gray}{0.9}

\newcommand{\expect}[1]{\mathbb{E} \!\left [ #1\right]}

\newcommand{\abs}[1]{\left\lvert #1 \right\rvert}

\newcommand{\innermid}{\nonscript\;\delimsize\vert\nonscript\;}
\newcommand{\activatebar}{%
  \begingroup\lccode`\~=`\|
  \lowercase{\endgroup\let~}\innermid 
  \mathcode`|=\string"8000
}

\newcommand{\C}{\mathbb{C}}

\newcommand{\N}{\mathbb{N}}

\newcommand{\R}{\mathbb{R}}

\newcommand{\PP}{\mathbb{P}}

\newcommand{\bcode}{\mathcal{B}}

\newcommand{\Linfty}{L^{\infty}}
\newcommand{\pvar}{p-\text{var}}

\newcommand{\suchthat}{\;\ifnum\currentgrouptype=16 \middle\fi|\;}

\DeclareMathOperator\csch{csch}
\DeclareMathOperator\sech{sech}

\DeclareMathOperator\erfc{erfc}

\DeclareMathOperator\sgn{sgn}

\DeclareMathOperator\rk{rank}

\DeclareMathOperator\Dgm{Dgm}

\DeclareMathOperator\Real{Re}

\DeclareMathOperator\updim{\overline{dim}}

\newcommand{\al}{\alpha}

\newcommand{\veps}{\varepsilon}
\newcommand{\vp}{\varphi}

\newcommand{\bd}{\begin{displaymath}
\begin{tikzcd}}
\newcommand{\ed}{\end{tikzcd}
\end{displaymath}}
\newcommand{\bmat}{\begin{pmatrix}}
\newcommand{\emat}{\end{pmatrix}}
\newcommand{\be}{\begin{equation}}
\newcommand{\ee}{\end{equation}}
\newcommand{\btikz}{\begin{tikzcd}}
\newcommand{\etikz}{\end{tikzcd}}
\newcommand{\bea}{\begin{eqnarray}}
\newcommand{\eea}{\end{eqnarray}}
\newcommand{\bse}{\begin{subequations}}
\newcommand{\ese}{\end{subequations}}
\newcommand{\bc}{\begin{center}}
\newcommand{\ec}{\end{center}}

\newcommand{\nonum}{\nonumber}

\newcommand{\half}{\frac{1}{2}}

\newcommand{\norm}[1]{\left\lVert#1\right\rVert}

\newcommand{\del}{\partial}

\newcommand{\comment}[1]{}

\newcommand{\cf}{{\it cf. }}

\newcommand{\Lag}{{\mathcal{L}}}

\def\blob[#1]{~\parbox{#1mm}{
\begin{fmfgraph*}(#1,#1)
\fmfleft{i1}
\fmfright{o1}
\fmf{phantom}{i1,v1,o1}
\fmfblob{0.4w}{v1}
\end{fmfgraph*}}~}
\def\vertex[#1]{~\parbox{#1mm}{
  \begin{fmfgraph*}(#1,#1)
    \fmfleft{i1, i2}
    \fmfright{o1,o2}
    \fmf{phantom}{i1,i2,o2,o1,i1}
    \fmf{plain}{i1,v,i2}
    \fmf{plain}{o1,v,o2}
    \fmfforce{nw}{i1}
    \fmfforce{sw}{i2}
    \fmfforce{se}{o1}
    \fmfforce{ne}{o2}
    \fmfforce{c}{v}
    \fmfdot{v}
  \end{fmfgraph*}
  }~}
\def\othervertex[#1]{~\parbox{#1mm}{
  \begin{fmfgraph*}(#1,#1)
    \fmfleft{i1, i2}
    \fmfright{o1,o2}
    \fmf{phantom}{i1,i2,o2,o1,i1}
    \fmf{plain}{i1,v}
    \fmf{plain}{i2,v}
    \fmf{plain}{o1,v}
    \fmf{plain}{v,o2}
    \fmfforce{nw}{i1}
    \fmfforce{sw}{i2}
    \fmfforce{se}{o1}
    \fmfforce{ne}{o2}
    \fmfforce{c}{v}
    \fmfblob{0.3w}{v}
  \end{fmfgraph*}
  }~}
\def\bigvertex[#1]{~\parbox{#1mm}{
  \begin{fmfgraph*}(#1,#1)
    \fmfleft{i1, i2}
    \fmfright{o1,o2}
    \fmf{phantom}{i1,i2,o2,o1,i1}
    \fmf{plain}{i1,v,i2}
    \fmf{plain}{o1,v,o2}
    \fmfforce{nw}{i1}
    \fmfforce{sw}{i2}
    \fmfforce{se}{o1}
    \fmfforce{ne}{o2}
    \fmfforce{c}{v}
    \fmfdot{v}
  \end{fmfgraph*}
  }~}
\def\edge[#1]{~\parbox{#1mm}{
  \begin{fmfgraph*}(#1,#1)
    \fmfleft{i}
    \fmfright{o}
    \fmf{plain,l.s=left}{i,o}
  \end{fmfgraph*}
}~}
\def\curlyedge[#1]{~\parbox{#1mm}{
  \begin{fmfgraph*}(#1,#1)
    \fmfleft{i}
    \fmfright{o}
    \fmf{curly,l.s=left}{i,o}
  \end{fmfgraph*}
}~}
\def\wavyedge[#1]{~\parbox{#1mm}{
  \begin{fmfgraph*}(#1,#1)
    \fmfleft{i}
    \fmfright{o}
    \fmf{wiggly,l.s=left}{i,o}
  \end{fmfgraph*}
}~}
\def\tadpole[#1]{~\parbox{#1mm}{
 	\begin{fmfgraph*}(#1,#1)
 		\fmfleft{i1}
 		\fmfright{o1}
		\fmf{plain}{i1,v1,v1,o1}
 		\fmfdot{v1}
 	\end{fmfgraph*}}~}
\def\amputatedtadpole[#1]{\begin{fmfgraph*}(#1,2)
		\fmfleft{i1}
		\fmfright{o1}
		\fmf{phantom}{i1,v1,o1}
		\fmf{plain}{v1,v1}
		\fmfdot{v1}
	\end{fmfgraph*}}
\def\amputatedwigglytadpole[#1]{\begin{fmfgraph*}(#1,2)
		\fmfleft{i1}
		\fmfright{o1}
		\fmf{phantom}{i1,v1,o1}
		\fmf{wiggly}{v1,v1}
		\fmfdot{v1}
	\end{fmfgraph*}}
\def\amputatedcurlytadpole[#1]{\begin{fmfgraph*}(#1,2)
		\fmfleft{i1}
		\fmfright{o1}
		\fmf{phantom}{i1,v1,o1}
		\fmf{curly}{v1,v1}
		\fmfdot{v1}
	\end{fmfgraph*}}
\def\vacfirstord[#1]{~\parbox{#1mm}{
 	\begin{fmfgraph*}(#1,#1)
		\fmfleft{i1,i2}
		\fmfright{o1,o2}
		\fmf{plain,left}{v2,v1}
		\fmf{plain,left}{v1,v2}
		\fmf{plain,left}{v3,v1}
		\fmf{plain,left}{v1,v3}
		\fmf{phantom}{v2,i1}
		\fmf{phantom}{v3,o1}
		\fmfforce{(0,0)}{i1}
		\fmfforce{(w,0)}{o1}
		\fmfforce{(0.5w,0.5h)}{v1}
		\fmfforce{(0,0.5h)}{v2}
		\fmfforce{(w,0.5h)}{v3}
		\fmfforce{nw}{i2}
		\fmfforce{ne}{o2}
		\fmfforce{(0,0.5h)}{i2}
		\fmfforce{(w,0.5h)}{o2}
		\fmfdot{v1}
		\end{fmfgraph*}}~}
\def\doubletadpolehor[#1]{~\parbox{#1mm}{
\begin{fmfgraph}(#1,#1)
	\fmfleft{i1}
	\fmfright{o1}
	\fmf{plain}{i1,v1,v1,v2,v2,o1}
	\fmfdot{v1,v2}
\end{fmfgraph}
}~}
\def\tripletadpolehor[#1]{~\parbox{#1mm}{
\begin{fmfgraph}(#1,#1)
	\fmfleft{i1}
	\fmfright{o1}
	\fmf{plain}{i1,v1,v1,v2,v2,v3,v3,o1}
	\fmfdot{v1,v2,v3}
\end{fmfgraph}
}~}
\def\doubletadpolever[#1]{~\parbox{#1mm}{
\begin{fmfgraph}(#1,#1)
	\fmfleft{i1}
	\fmfright{o1}
	\fmf{plain}{i1,v1}
	\fmf{plain,left}{v1,v2}
	\fmf{plain,left}{v2,v1}
	\fmf{plain}{v1,o1}
	\fmffreeze
	\fmf{plain,left}{v2,v3,v2}
	\fmfforce{c}{v2}
	\fmfforce{(0.5w,h)}{v3}
	\fmfforce{(0.5w,0)}{v1}
	\fmfforce{sw}{i1}
	\fmfforce{se}{o1}
	\fmfdot{v1,v2}
\end{fmfgraph}
}~}
\def\amputateddoubletadpolever[#1]{~\parbox{#1mm}{
	\begin{fmfgraph}(#1,#1)
	\fmfleft{i1}
	\fmfright{o1}
	\fmf{phantom}{i1,v1}
	\fmf{plain,left}{v1,v2}
	\fmf{plain,left}{v2,v1}
	\fmf{phantom}{v1,o1}
	\fmffreeze
	\fmf{plain,left}{v2,v3,v2}
	\fmfforce{c}{v2}
	\fmfforce{(0.5w,h)}{v3}
	\fmfforce{(0.5w,0)}{v1}
	\fmfforce{sw}{i1}
	\fmfforce{se}{o1}
	\fmfdot{v1,v2}
	\end{fmfgraph}~
}
}
\def\sunset[#1]{\parbox{#1mm}{
\begin{fmfgraph}(#1,#1)
\fmfleft{i}
\fmfright{o}
\fmfforce{0,0.5h}{i}
\fmfforce{w,0.5h}{o}
\fmf{plain,tension=5}{i,v1}
\fmf{plain,tension=5}{v2,o}
\fmf{plain,left,tension=0.5}{v1,v2,v1}
\fmf{plain}{v1,v2}
\fmfdot{v1,v2}
\end{fmfgraph}
}}
\def\amputatedsunset[#1]{\parbox{#1mm}{
\begin{fmfgraph}(#1,#1)
\fmfleft{i}
\fmfright{o}
\fmf{phantom,tension=5}{i,v1}
\fmf{phantom,tension=5}{v2,o}
\fmf{plain,left,tension=0.5}{v1,v2,v1}
\fmf{plain}{v1,v2}
\fmfdot{v1,v2}
\end{fmfgraph}
}}
\def\fourpointsecondorder[#1]{~\parbox{#1mm}{
	\begin{fmfgraph}(15,#1)
	\fmfleft{i1,i2}
	\fmfright{o1,o2}
	\fmf{plain}{i1,v1}
	\fmf{plain}{i2,v1}
	\fmf{plain,right}{v1,v2}
	\fmf{plain,left}{v1,v2}
	\fmf{plain}{v2,o1}
	\fmf{plain}{v2,o2}
	\fmfdot{v1,v2}
	\end{fmfgraph}}~
}
\def\fourpointsecondordertwo[#1]{~\parbox{#1mm}{
	\begin{fmfgraph}(15,#1)
	\fmfleft{i1,i2,i3}
	\fmfright{o1}
	\fmf{plain}{i1,v1}
	\fmf{plain}{i2,v1}
	\fmf{plain}{i3,v1}
	\fmf{plain}{v1,v2,v2,o1}
	\fmfdot{v1,v2}
	\end{fmfgraph}
}~}

\newtheorem{theorem}{Theorem}[section]

\catcode`,\active

\catcode`\,12

\theoremstyle{definition}
\newtheorem{THM}{Theorem}
\newtheorem{COR}{Corollary}

\newtheorem{definition}[theorem]{Definition}

\newtheorem{lemma}[theorem]{Lemma}
\newtheorem{proposition}[theorem]{Proposition}
\newtheorem{defprop}[theorem]{Definition/Proposition}

\theoremstyle{remark}
\newtheorem{remark}[theorem]{Remark}
\theoremstyle{example}

\newtheorem{corollary}[theorem]{Corollary}

\title{On the persistent homology of almost surely $C^0([0,t],\R)$ stochastic processes}

\author[1,2,3]{Daniel Perez\thanks{Email: \texttt{daniel.perez@ens.fr}}}
\affil[1]{\footnotesize D\'epartement de math\'ematiques et applications, \'Ecole normale sup\'erieure, CNRS, PSL University, 75005 Paris, France}
\affil[2]{\footnotesize Laboratoire de math\'ematiques d'Orsay, Universit\'e Paris-Saclay, CNRS, 91405 Orsay, France}
\affil[3]{\footnotesize DataShape, Centre Inria Saclay, 91120 Palaiseau, France}

\date{\today}

\begin{document}
\maketitle

\begin{abstract}
This paper investigates the propreties of the persistence diagrams stemming from almost surely continuous random processes on $[0,t]$. We focus our study on two variables which together characterize the barcode : the number of points of the persistence diagram inside a rectangle $]\!-\!\infty,x]\times [x+\veps,\infty[$, $N^{x,x+\veps}$ and the number of bars of length $\geq \veps$, $N^\veps$. For processes with the strong Markov property, we show both of these variables admit a moment generating function and in particular moments of every order. Switching our attention to semimartingales, we show the asymptotic behaviour of $N^\veps$ and $N^{x,x+\veps}$ as $\veps \to 0$ and of $N^\veps$ as $\veps \to \infty$.  Finally, we study the repercussions of the classical stability theorem of barcodes and illustrate our results with some examples, most notably Brownian motion and empirical functions converging to the Brownian bridge.
\end{abstract}

\tableofcontents

\section{Introduction}
\label{intro}
\subsection{State of the art}
One of the questions of interest in the theory of persistent homology is the following: given a random function on some topological space $X$, what can we say about the barcode $\bcode(X)$ of this process? The study of the topology of (super)level-sets of random functions has been a subject of interest in probability theory for a long time \cite{RiceFormula,AdlerTaylor:RandomFields,AzaisWschebor,Kahane,LeGall:BrownianMotion,BMPeresMorters}. Most prominently for this paper, by Le Gall and Duquesne, who gave a construction of a tree from any continuous function $f: [0,1] \to \R$ \cite{LeGall:Trees}, and who interpreted different properties of these trees to give fine results about L\'evy processes \cite{LeGallDuquesne:LevyTrees}. Picard later linked the upper-box dimension of these trees to the regularity of the function $f$ \cite{Picard:Trees}. In essence, these trees have proved to be a fruitful and natural setting from which many results regarding the topology of the superlevel sets of the function $f$ stem \cite{curien2013,Evans:RealTrees,Neveu_1989, Curry_2018, Perez_2020}.

A natural question is whether, or indeed how, these results are applicable to the persistent homology of stochastic processes. The answer turns out to be total: the study of barcodes and trees are completely equivalent in degree $0$ of homology. This has been established in \cite{Perez_2020}, in which a dictionary between $H_0$-barcodes and the trees of Le Gall and Duquesne was constructed.

In this paper, we focus mainly on two barcode-related quantities, namely the number of bars in the barcode of $f$ of length $\geq \veps$, which we will denote $N_f^\veps$ and the number of bars in the barcode of $f$ whose intersection with the interval $[x,x+\veps]$ is non-empty, which we will denote $N^{x,x+\veps}_f$. Whenever the function is implicit from the context, we may omit the subscript $f$.

With the dictionary established in \cite{Perez_2020}, it is possible to interpret the results previously obtained by Picard in the context of trees. Most notably, Picard showed that the regularity of functions is closely related to their small-bar asymptotics.
\begin{theorem}[Picard, \S 3\cite{Picard:Trees}]
\label{thm:pvarandupboxdim}
Given a continuous function $f: [0,1] \to \R$, 
\be
\mathcal{V}(f) = \updim T_f = \limsup_{\veps \to 0} \frac{\log N^\veps_f}{\log(1/\veps)} \vee 1
\ee
where $\updim$ denotes the upper-box dimension, $a \vee b = \max\{a,b\}$,
\be
\mathcal{V}(f) := \inf\{p \geq 1\, \vert \, \norm{f}_{\pvar} <\infty\} \,.
\ee
\end{theorem}
Moreover, for processes which are self-similar in distribution, Picard \cite{Picard:Trees} shows the almost sure small-bar asymptotics is closely related to this self-similarity (the results are shown for trees, but are immediately interpretable in terms of barcodes by the results of \cite{Perez_2020}). 
\begin{theorem}[Picard, \S 3\cite{Picard:Trees}]
\label{thm:pvarandupboxdim}
Let $X: [0,1] \to \R$ be Brownian motion, a Lévy $\al$-stable process or fractional Brownian motion, then, almost surely, there exists a constant $C_H$ such that
\be
N^\veps_X \sim \frac{C_H}{\veps^{1/H}} \quad \text{as } \veps \to 0 \,,
\ee
where $H$ is the self-similary index, \textit{i.e.} the $H$ for which $X$ satisfies $X_{\lambda t} = \lambda^H X_t$ for every $\lambda >0$ and $t\geq 0$ in distribution.
\end{theorem}
Parallel to these developments, some results regarding the persistent homology of Brownian motion have also been provided by the topological data analysis (TDA) community. In particular, for Brownian motion $B$ Chazal and Divol gave a formula for the distribution of the number of points $N^{x,y}_B$ lying inside a given rectangle $]\!-\!\infty,x]\times [y,\infty[$ in the persistence diagram of Brownian motion, $\Dgm(B)$ \cite{ChazalDivol:Brownian}.
\begin{proposition}[Chazal, Divol, \cite{ChazalDivol:Brownian}]
\label{prop:ChazalDivol}
For $0<x<y$, the distribution of $N^{x,y}_B$ is
\be
\PP(N^{x,y}_B \geq k) = \int_{\Sigma_{2k-1}} \psi(x,t_1) \psi(y-x, s_1)\psi(y-x,t_2)\cdots \psi(y-x, t_k) \prod_{i=1}^{k-1}dt_i \prod_{j=1}^{k-2} ds_j \,,
\ee
where $\Sigma_{2k-1}$ denotes the corner of the domain bounded by the $(2k-1)$-simplex and we note a vector in $\R^{2k-1}$ by $(t_1, s_1,\cdots, s_{k-2},t_{k-1})$ and 
\be
\psi(x,t) := \frac{x}{\sqrt{2\pi t^3}} e^{-\frac{x^2}{2t}} \;.
\ee
\end{proposition}
Using this result Chazal and Divol established that $\expect{N^{x,y}}$ was $C^1$ in $x$ and $y$. In the context of our own results, we will show this relation to be analytic and give an explicit expression for $N^{x,x+\veps}$ for $x>0$ and $\veps>0$.

Similar results were obtained by Baryshnikov \cite{Baryshnikov_2019}, who computed exactly $\expect{N^{x,y}_{B^\mu}}$ for the Brownian motion with a strictly positive linear drift $B^\mu_t:= \mu t + B_t$ for $\mu>0$ over the whole ray $[0,\infty[$.
\begin{proposition}[Baryshnikov, \cite{Baryshnikov_2019}]
\label{prop:Baryshnikov}
The expected value of $N^{x,x+\veps}_{B^{\mu}}$ for $x>0$ and $\mu>0$ over the whole ray $[0,\infty[$ is given by
\be
\expect{N^{x,x+\veps}_{B^{\mu}}} = \frac{1}{e^{2\mu\veps}-1} \,.
\ee
In particular, as $\veps \to 0$, the following asymptotic relation holds
\be
\expect{N^{x,x+\veps}_{B^{\mu}}} \sim \frac{1}{2\mu\veps} - \half + \frac{1}{6}\mu \veps + O(\mu^3\veps^3) \quad \text{as } \veps \to 0 \,.
\ee
\end{proposition}

In this paper, we will be concerned with almost surely $C^0$ processes, but it is noteworthy that the study of smooth random fields (and their topology) is currently an open field of research in probability theory. A good introduction to the smooth setting is provided by the celebrated books of Adler and Taylor \cite{AdlerTaylor:RandomFields}, and that of Aza\"is and Wschebor \cite{AzaisWschebor}.

Whenever possible and necessary, we give the definitions and most important results necessary to make the proofs in this paper self-contained. However, for the sake of brevity, we do not introduce \textit{all} the probabilistic concepts necessary for this paper (most notably, we do not define Brownian motion or stochastic processes, filtrated probability spaces, nor stopping times and the strong Markov property). We kindly refer the reader unfamiliar with these concepts to Le Gall's \cite{LeGall:BrownianMotion} and Revuz and Yor's \cite{Revuz_1999} books on stochastic calculus for a comprehensive introduction to the subject.

\subsection{Our contribution}
Our contribution can be summarized along the following points. 

\begin{THM}[Theorem \ref{thm:MarkovNvepsRange}]
let $X$ be an almost surely continuous stochastic process on $[0,t]$ with the strong Markov property. Define its range to be the random variable 
\be
R_t =\sup_{[0,t]} X -\inf_{[0,t]} X \,.
\ee
Assume there exist $\veps^*$ such that $P(R_t>\veps^*)<1$. Then for all $\veps>\veps^*$, $N^\veps_X$ and $N^{x,x+\veps}_X$ have moments of all order and admit a moment generating function.

\end{THM}
This theorem answers the question of the existence of moments (and their quantification) for the quantities $N^\veps$ and $N^{x,x+\veps}$ in the context of almost surely continuous Markov processes. It also allows us to show the behaviour of large bars of the barcode of such a process in expectation.
\begin{COR}[Corollary \ref{cor:vepstoinfty}]
Under the same assumptions,
\be
\expect{N^\veps_X} \sim \PP(R_t \geq \veps) \quad \text{as } \veps \to \infty \,. 
\ee

\end{COR}
Having answered the question of large bars (at least partially) in this very general context, we switch our attention to the behaviour of small bars. Indeed, as the following result shows, the latter is very regular and closely related to the nature of the noise at hand.
\begin{THM}[Theorem \ref{thm:semimartingales}]
Let $X=M+A$ be a continuous semimartingale on $[0,t]$ and suppose that for $s \geq 1$ 
\be
\expect{[M]_t^{s/2} + \left(\int_0^t \abs{dA}_s\right)^s} < \infty \,.
\ee 
Then in $L^s(\Omega)$, 
\begin{align*}
N_X^{x,x+\veps} &\sim \frac{L^x_X(t)}{2\veps} + O(1) \quad \text{as } \veps \to 0 \, \\
N_X^\veps &\sim \frac{[X]_t}{2\veps^2} + O(\veps^{-1}) \, \quad \text{as } \veps \to 0 \,,
\end{align*}
where $L^x_X(t)$ denotes the local time of $X$ on $[0,t]$ at $x$.
\end{THM}
For particular instances of local martingales, we have a full description of the barcode of these processes, which depends on their quadratic variation.
\begin{THM}[Theorem \ref{thm:locmart}]
For any continuous local martingale $M$ on $[0,t]$ having deterministic and strictly increasing quadratic variation $[M]_t$ such that $[M]_\infty = \infty$, 
\begin{align*}
\expect{N^\veps_M} &= 4 \sum_{k\geq 1} (2k-1)\erfc\!\left(\frac{(2k-1)\veps}{\sqrt{2[M]_t}}\right) - k\,\erfc\!\left(\frac{2k\veps}{\sqrt{2[M]_t}}\right) \\
&= \frac{[M]_t}{2\veps^2} +\frac{2}{3} + 2\sum_{k\geq 1} (2(-1)^k -1) \frac{e^{-\pi^2 k^2 [M]_t/2\veps^2}[M]_t}{\veps^2}  \left[1 + \frac{\veps^2}{\pi^2k^2[M]_t}\right]\,.
\end{align*}
Moreover on $[0,t]$ and for $x>0$,
\begin{align*}
\expect{N^{x,x+\veps}_M} &= \sum_{k=1}^\infty \erfc\left(\frac{x+(2 k-1)\veps}{\sqrt{2[M]_t}}\right) \\
&\sim \frac{1}{2\veps} \int_0^{[M]_t} \vp(x,s) \;ds +\sum_{k\geq 0} \frac{4 (-2)^{k}\!\left(2^{2k+1}-1\right) \zeta(2k+2)}{\pi^{2k+2}} \left[\frac{\del^{k}}{\del s^{k}}\Big\vert_{s=[M]_t} \vp(x,s)\right]\veps^{2k+1}  \; \text{as }\veps \to 0 \,,
\end{align*}
where $\vp(x,t)$ is the density of a centered Gaussian random variable of variance $t$ and $\zeta$ denotes the Riemann zeta function.
\end{THM}
This theorem allows us to explain exactly the experimental observations we made through simulation regarding the barcode of Brownian motion (\cf figures \ref{fig:NvepsBMnothing} and \ref{fig:BMCalculated}).
\begin{figure}[h!]
  \centering
  \includegraphics[width=0.65\textwidth]{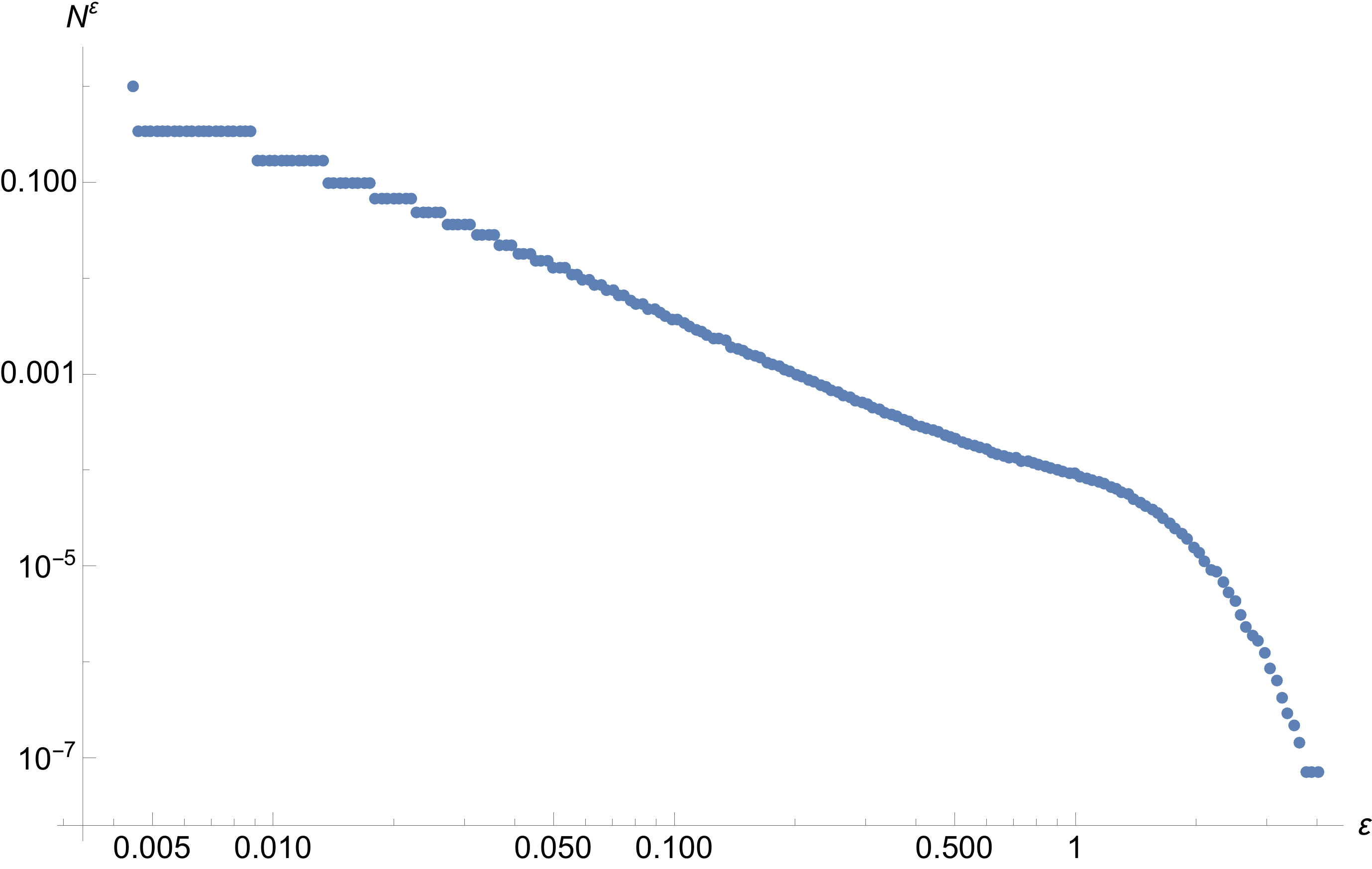}
  \caption{$\expect{N^\veps_B}$ as approximated with simulations of random Rademacher walks.}
  \label{fig:NvepsBMnothing}
\end{figure}

We use the classical stability theorem of barcodes \cite{Chazal:Persistence} and complete our description of random processes by examining sequences converging to the processes above, which constitute universal limits for many random processes (most notably, random walks and empirical processes). 
\begin{THM}[Propositions \ref{prop:btsnbtmatching} and \ref{prop:LpLinftyCvgBars}]
Let $(M,d)$ be a compact Polish metric space and let $X$ be an almost surely continuous stochastic process on $M$, defined on a probability space $(\Omega, \mathcal{F},\PP)$. Let $(X_n)_{n\in \N}$ be any sequence of continuous stochastic processes defined on the same probability space and suppose there exists a $p \geq 1$ such that
\be
\delta_n := \norm{X-X_n}_{L^p(\Omega, L^\infty(M,\R))} \xrightarrow[n\to \infty]{} 0 \,.
\ee 
Then, with probability $\geq 1- \frac{1}{a^p}$, for every $\veps \geq 2a\delta_n$,
\be
\abs{N^{\veps}_{X_n}- N^{\veps}_{X}} \leq N_X^{\veps-2a\delta_n} - N_X^{\veps +2a\delta_n} \,.
\ee
Suppose further that $\expect{N_X^\veps}$ is continuous in $\veps$. Then,
\be
\expect{\abs{N_{X_n}^\veps - N_X^\veps} \, \Big \vert \, \norm{X-X_n}_{\infty} \leq a\delta_n} \leq \omega_\veps(2a\delta_n) \,.
\ee
where $\omega_\veps(\delta) := \expect{N_X^{\veps-\delta}- N_X^{\veps+\delta}}$. Moreover, 
\be
\PP\!\left(\abs{N_{X_n}^\veps - N_X^\veps} \geq k \right) \leq \frac{\omega_\veps(2a \delta_n)}{k} + \frac{1}{a^p} \quad \text{and} \quad  \PP\!\left(\abs{N_{X_n}^\veps - N_X^\veps} \geq k \; , \; \norm{X-X_n}_{\infty} \leq a\delta_n\right) \leq \frac{\omega_\veps(2a \delta_n)}{k}\,.
\ee
If $p=\infty$ and still assuming $\expect{N^\veps_X}<\infty$ and is continuous, for any $\veps \geq 2\delta_n$,
\be
N_{X_n}^{\veps} \xrightarrow[n\to \infty]{L^1} N_X^{\veps} \quad \text{and} \quad \expect{\abs{N_X^\veps- N_{X_n}^\veps}} \leq \omega_\veps(2\delta_n)  \,. \nonum
\ee
With analogous hypotheses, the same statement holds for $N^{x,x+\veps}_X$.
\end{THM}
Finally, we use the previous results to give statements about well-known processes, such as Brownian motion, Itô processes and some limiting processes (most notably the Fourier decomposition of Brownian motion and the convergence of empirical processes to the Brownian bridge), \cf section \ref{sec:Applications}.

\section{Some generalities about $H_0$ homology and trees}
Let us briefly recall the construction of a tree from a continuous function $f: M \to \R$ detailed in \cite{Perez_2020}.
\begin{defprop}
Let $M$ denote a connected, locally path-connected, compact topological space, $f: M \to \R$ be a continuous function and let $x,y \in X$, then the function
\be
d_f(x,y) := f(x) +f(y) -2 \sup_{\gamma: x \mapsto y}\min_{t \in [0,t]} f(\gamma(t)) \,,
\ee
where the supremum runs over all paths $\gamma: [0,1] \to M$ is a pseudo-distance on $M$ and the quotient metric space
\be
T_f := M/\{x \sim y \iff d_f(x,y) =0\}
\ee
equipped with the distance $d_f$ is a rooted $\R$-tree, whose root coincides with the image in $T_f$ of the point in $[0,1]$ at which $f$ achieves its infimum. 
\end{defprop}
\begin{definition}
Let $M$ denote a connected, locally path-connected, compact topological space and $f: M \to \R$ be a continuous function. We denote $\pi_f : M \to T_f$ the canonical projection.
\end{definition}
\begin{figure}[h!]
  \centering
    \includegraphics[width=0.7\textwidth]{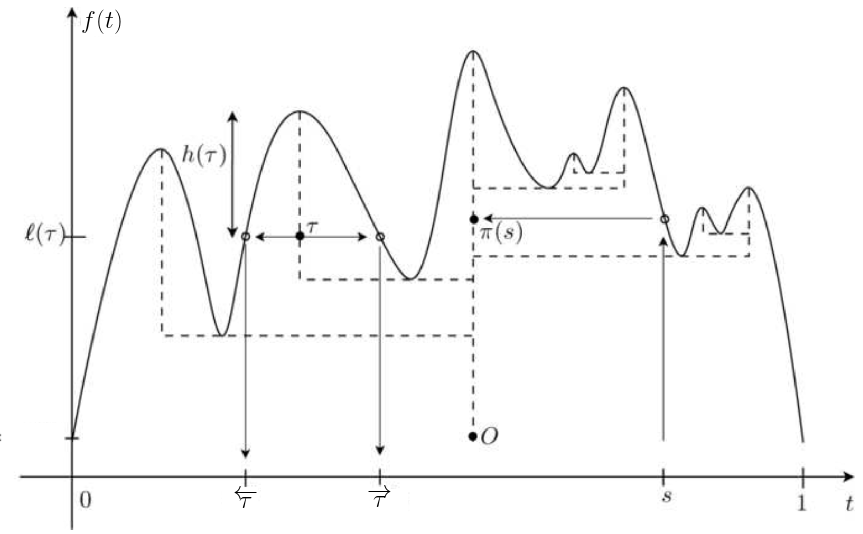}
  \caption{A function $f: [0,1] \to \R$ and its associated tree $T_f$ in dashed lines.}
\label{fig:treequantities}
\end{figure}
The tree $T_f$ has the particularity that its branches correspond to connected components of the superlevel sets of $f$, as illustrated by figure \ref{fig:treequantities}. 
To define $N^\veps$ on this tree, it is first necessary to introduce the so-called $\veps$-simplified or $\veps$-trimmed tree of $T_f^\veps$.  This object is obtained by ``giving a haircut'' of length $\veps$ to $T_f$. More precisely, if we define a function $h: T_f \to \R$ which to a point $\tau \in T_f$ associates the distance from $\tau$ to the highest leaf above $\tau$ with respect to the filtration on $T_f$ induced by $f$, then
\begin{definition}
Let $\veps \geq 0$. An \textbf{$\veps$-trimming} or \textbf{$\veps$-simplification} of $T_f$ is the metric subspace of $T_f$ defined by
\be
T_f^\veps := \{\tau \in T_f \,\vert \, h(\tau) \geq \veps\}
\ee
\end{definition}
With this definition, we can interpret $N^\veps$ geometrically as being equal to the number of leaves of $T_f^\veps$. The reason for this is explicited in \cite{Perez_2020}.
The idea is that, starting from $T_f$, we can look at the longest branch (starting from the root) of $T_f$. This branch corresponds to the longest bar of $\bcode(f)$, since branches of $T_f$ correspond to connected components of the superlevel sets of $f$. Next, we erase  this longest branch and, on the remaning (rooted) forest, look for the next longest branch. This will be the second longest bar of the barcode. Proceeding iteratively in this way, we retrieve $\bcode(f)$. An illustration of this can be found in figure \ref{fig:algorithm}.
\begin{figure}[h!]
  \centering
    \includegraphics[width=0.6\textwidth]{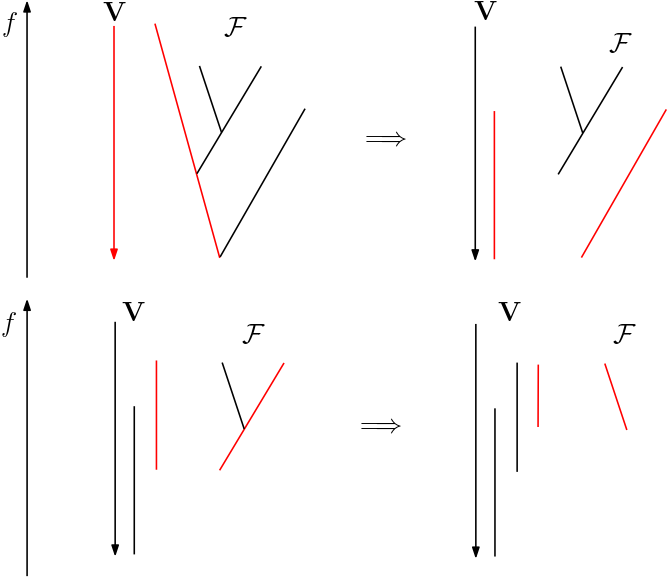}
  \caption{A depiction of the first steps of the algorithm which assigns a barcode $\bcode(f)$ to a tree $T_f$.}
\label{fig:algorithm}
\end{figure}
\begin{defprop}
The number of bars alive at $x$ persisting through $x+\veps$, 
\be
N^{x,x+\veps}_f := \rk(H_0(\{f\geq x+\veps\} \to \{f\geq x\})) = \#\{\tau \in T_f \, \vert\,  f(\tau) =x \text{ and } h(\tau) \geq \veps\} \,.
\ee
\end{defprop}
\begin{proof}
By the results of \cite{Perez_2020}, there is a bijective correspondence between points in $T_f$ and points in the barcode (as seen as points in the collection of intervals of the barcode). It follows that every element of the set 
\be
\{\tau \in T_f \, \vert\,  f(\tau) =x \text{ and } h(\tau) \geq \veps\}  \,,
\ee
has a corresponding image in one and only one bar of the barcode, whose intersection with $[x,x+\veps]$ is not empty. Conversely, for every bar in the barcode with non-empty intersection with the interval $[x,x+\veps]$ there is a unique image on the tree at height $x$, $\tau$, and, since the bar persists a length $\geq \veps$ after $x$, $h(\tau) \geq \veps$.
\end{proof}

\begin{remark}[Link with traditional persistence diagrams and sublevel set filtrations]
In the tree formalism, it is typical to consider \textbf{superlevel} filtrations as opposed to sublevel ones, as is typically done in traditional persistence theory. Considering one or the other poses no problem for us, as one can pass from one filtration to the other by switching $f$ into $-f$. Diagrams (in the sense of collections of points $(b,d)$ of moments of birth $b$ and death $d$ of bars of the barcode) of a filtration by superlevel sets lie \textit{below} the diagonal, as the moment of birth occurs \textit{higher} than the moment of death.
Given the diagram of $-f$ as computed with the \textit{superlevel} filtration (using for instance the tree), we can retrieve the diagram associated to \textit{sublevels} of $f$ by sending each point in the diagram (of the superlevel filtration of $-f$) 
\be
(b,d) \mapsto (-d,-b) \,.
\ee 
For most results, this subtlety makes little to no difference, as most of the examples considered (typically Brownian motion $B$) are processes which are symmetric, \textit{i.e.} $B_t = - B_t$ in distribution for all $t$. This is notably the case for the results obtained by Chazal and Divol \cite{ChazalDivol:Brownian}.

\noindent Finally, with respect to the superlevel (resp. sublevel) set filtration, we will henceforth \textbf{always} consider that the infinite bars of the barcode are capped at $\inf(f)$ (resp. $\sup(f)$). Equivalently, in terms of barcodes we will always consider for any bar 
\be
b = b \cap [\inf f,\sup f] \,.
\ee
\end{remark}

On a tree $T_f$, we can define a notion of integration by defining the unique atomless Borel measure $\lambda$ which is characterized by the property that every geodesic segment on $T_f$ has measure equal to its length. Formally, we can express $\lambda$ in two ways \cite{Picard:Trees}
\be
\lambda = \int_\R dx \sum_{\substack{\tau \in T_f \\ f(\tau) = x}} \delta_\tau \quad \text{and} \quad \lambda = \int_0^\infty d\veps \sum_{\substack{\tau \in T_f \\ h(\tau) = \veps}} \delta_\tau \,.
\ee
\begin{proposition}
\be
\lambda(T_f^\veps) =  \int_\veps^\infty N^a_f \;da = \int_\R N^{x,x+\veps}_f \;dx
\ee
\end{proposition}
\begin{proof}
By using the second identity for $\lambda$,
\be
\lambda(T_f^\veps) = \int_\veps^\infty N^a_f\;da \,,
\ee
since every sum in the second expression is finite for all $\veps >0$ and has $N^\veps$ terms. Writing it using the first identity, we must restrict the sum in the identity to
\be
\sum_{\substack{\tau \in T_f \\ f(\tau)=x \\ h(\tau)\geq\veps}} \delta_\tau \,,
\ee
which is finite for all $\veps >0$. There are exactly $N^{x,x+\veps}_f$ terms in this sum, therefore
\be
\lambda(T_f^\veps)= \int_\R  N^{x,x+\veps}_f \;dx\,.
\ee
\end{proof}


\subsection{Persistent homology of processes on $[0,t]$}
Consider now $f: [0,t] \to \R$. Given the total order on $\R$, the preimages of $T_f$ by $\pi_f: [0,t] \to T_f$ inherit a natural order structure. This allows us to define
\begin{definition}
The \textbf{right} (resp. \textbf{left}) \textbf{preimage of $\tau \in T_f$ by $\pi_f$} is
\be
\overrightarrow{\tau} := \sup \pi_f^{-1}(\tau)  \quad \text{(resp. } \overleftarrow{\tau} := \inf \pi_f^{-1}(\tau) \text{)} \,.
\ee
\end{definition}
\begin{remark}
Notice we can have $\overleftarrow{\tau}= \overrightarrow{\tau}$. See figure \ref{fig:treequantities} for a depiction of these preimages.
\end{remark}

Whenever $f:[0,t] \to \R$, we can compute $N^\veps_f$  is by counting the number of times we go up by at least $\veps$ from a local minimum and down by at least $\veps$ from a local maximum. This idea can be formalized by the following sequence, originally introduced by Neveu \textit{et al.} \cite{Neveu_1989}.
\begin{definition}
\label{def:vepsminmax}
Setting $S_0^\veps = T_0^\veps = 0$, we define a sequence of times by induction
\begin{align*}
T_{i+1}^\veps &:= \inf\left\{\, s \geq S_i \; \Bigg\vert \; \sup_{[S_i^\veps, s]} f -f(s) > \veps\right\} \\
S_{i+1}^\veps &:= \inf\left\{\, s \geq T_{i+1}\; \Bigg\vert \; f(s) - \inf_{[T_{i+1}^\veps, s]} f > \veps\right\}
\end{align*}
\end{definition}

\begin{lemma}
\label{lemma:NvepsSk}
If $k \geq 2$, and $f: [0,t] \to \R$ is continuous,
\be
N^\veps_f \geq k \iff S_{k-1}^\veps \leq t \,.
\ee
\end{lemma}
\begin{proof}
$(N^\veps \geq k \implies S_{k-1}^\veps \leq t)$: We start by noticing we can order the bars by the value of their preimages by virtue of the total order on $\R$. Since $N^\veps \geq k$, there are at least $(k-1)$ right preimages and left preimages by $\pi_f$ of leaves of $T_f^\veps$ stemming from the first $(k-1)$ bars, which we denote $\{(\overleftarrow{\tau}_i,\overrightarrow{\tau}_i)\}_{1 \leq i \leq k-1}$ satisfying 
\be
\overleftarrow{\tau}_1 \leq \overrightarrow{\tau}_1 = T^\veps_1  < S_{1}^\veps \leq \overleftarrow{\tau}_2\leq \overrightarrow{\tau}_2 = T^\veps_2 < \cdots \leq \overleftarrow{\tau}_{k-1} \leq \overrightarrow{\tau}_{k-1} = T^\veps_{k-1} \,.
\ee
Note that $\overrightarrow{\tau}_1>0$ as soon as $k \geq 2$ and $\veps>0$. But $N^\veps \geq k$, therefore there must exist a preimage $\overleftarrow{\tau}_k \leq t$ corresponding to the $k$th bar. But $S_{k-1}^\veps \leq \overleftarrow{\tau}_k \leq t$ by definition of $S_{k-1}^\veps$. 
\par
$(N^\veps \geq k \Longleftarrow S_{k-1}^\veps \leq t)$: Since $S_{k-1}^\veps > T_{k-1}^\veps = \overrightarrow{\tau}_{k-1}$, there are at least $(k-1)$ distinct bars. We easily check that, by definition, $d_f(S_{k-1}^\veps,T_{k-1}^\veps) >0$, implying that $\pi_f(S_{k-1}^\veps)$ and $\pi_f(T_{k-1}^\veps)$ lie on different branches of $T_f$, and therefore $S_{k-1}^\veps$ is a preimage of a distinct bar of length $\geq \veps$, implying $N^\veps \geq k$.
\end{proof}

\begin{definition}
\label{def:vepsminmax}
Let $f : [0,t] \to \R$ be a continuous function. Setting $US_0^{x,\veps} = UT_0^{x,\veps} = 0$, we define a sequence of times recursively
\begin{align}
UT_{i+1}^{x,\veps} &:= \inf\left\{\, s \geq US_i^{x,\veps} \; \Bigg\vert \; f(s) \leq x \wedge (x+\veps) \right\} \nonum\\
US_{i+1}^{x,\veps} &:= \inf\left\{\, s \geq UT_{i+1}^{x,\veps}\; \Bigg\vert \; f(s) \geq x \vee (x+\veps) \right\} \,. \nonum
\end{align}
The maximum $i$ for which $US_i^{x,\veps} \leq t$ is called the \textbf{number of upcrossings by $f$ from $x$ to $x+\veps$} and we denote it $U^{x,x+\veps}_f$. Similarly, setting $DS_0^{x,\veps} = DT_0^{x,\veps} = 0$ and defining
\begin{align}
DT_{i+1}^{x,\veps} &:= \inf\left\{\, s \geq DS_{i+1}^{x,\veps} \; \Bigg\vert \; f(s) \leq x \wedge (x+\veps) \right\} \nonum\\
DS_{i+1}^{x,\veps} &:= \inf\left\{\, s \geq DT_{i}^{x,\veps}\; \Bigg\vert \; f(s) \geq x \vee (x+\veps) \right\} \,, \nonum
\end{align}
we can define the \textbf{number of downcrossings by $f$ from $x$ to $x+\veps$} and denote it $D^{x,x+\veps}_f$ as the maximum $i$ for which $DT_i^{x,\veps} \leq t$.
\end{definition}
\begin{proposition}
\label{prop:NxxvepsDowncrossings}
Let $f : [0,t] \to \R$ be a continuous function. Then,
\be
N^{x,x+\veps}_f = U_f^{x,x+\veps} \vee D_f^{x,x+\veps} \leq D^{x,x+\veps}_f + 1\,.
\ee
If $x\geq 0$, and $f(0)=0$, $N^{x,x+\veps}_f =U_f^{x,x+\veps}$.
\end{proposition}
\begin{proof}
Each bar alive at $x+\veps$ persisting until $x$ (seen as a branch of $T_f$) admits a left and right preimage (which can be sometimes equal), this shows $N^{x,x+\veps} \geq U_f^{x,x+\veps} \vee D_f^{x,x+\veps}$. Conversely, every downcrossing (resp. upcrossing) generates a distinct homology class alive at $x+\veps$ persisting until $x$, which shows $U_f^{x,x+\veps} \vee D_f^{x,x+\veps} \geq N^{x,x+\veps}$. The last inequality is a consequence of the fact that, by definition $\abs{U^{x,x+\veps}_f - D^{x,x+\veps}_f} \leq 1$. Finally, if $x\geq 0$, and $f(0)=0$, by continuity of $f$, there is a bijective correspondence between upcrossings and the bars alive at $x+\veps$ persisting until $x$.
\end{proof}

\begin{theorem}
\label{thm:MarkovNvepsRange}
Let $X$ be a non-constant stochastic process on $[0,t]$ defined on the filtered probability space $(\Omega, \mathcal{F},\PP)$ allowing an almost surely continuous modification and satisfying the strong Markov property. Denote 
\be
R_t := \sup_{[0,t]} X - \inf_{[0,t]} X \,,
\ee
Then,
\be
\PP(N^\veps \geq k) \leq \PP(R_t \geq \veps)^{2(k-1)} \,.
\ee
Suppose further that $X$ that there exists some $\veps^*$ such that for all $\veps > \veps^*$, $\PP(R_t \geq \veps^*) <1$.
Then for every $\veps >\veps^*$ and every $x \in \R$, all the moments of the random variables $N^\veps$ and $N^{x,x+\veps}$ are finite and their moment generating functions $M(\lambda)$ converge uniformly and absolutely for every $\lambda \in \C$ such that $\Real(\lambda) > 2 \log(\PP(R_t \geq \veps))$.
\end{theorem}

\begin{proof}
For simplicity, let us set $t =1$. The probability $\PP(N^\veps \geq k)$ can be written in terms of the stopping times $T_i^\veps$ and $S_i^\veps$ and their increments by the (strong) Markov property of $X$. By lemma \ref{lemma:NvepsSk}, $\PP(N^\veps \geq k) = \PP(S_{k-1}^\veps \leq 1)$ for $k\geq 2$, so
\be
\PP(N^\veps \geq k) = \int_{\Sigma_{2k-2}}\PP(T_1^\veps = t_1)\PP(S_1 =s_1 \, \vert \, T_1 = t_1) \cdots \PP(S_{k-1}=s_{k-1} \,\vert\, T_{k-1}=t_{k-1}) \prod_{i=1}^{k-1}ds_i dt_i  \,.
\ee
where $\Sigma_{2k-2}$ denotes the simplex
\be
\Sigma_{2k-2} := \left\{(t_1, s_1,\cdots,s_{k-1}) \in \R^{2k-1} \; \big\vert \; 0 \leq t_1 \leq s_2 \leq \cdots \leq s_{k-1} \leq 1\right\} \,.
\ee
By the definition of these stopping times we know that
\begin{align*}
\PP(s\leq T_i^\veps \leq t \vert S_{i-1}= s)   &=  \PP\!\left(\sup_{\tau \in [s,t]} \left[ \sup_{[s,\tau]} X - X_\tau \right] \geq \veps \right) \\
\PP(t \leq S_i^\veps \leq s \vert T_{i} = t) &= \PP\!\left(\sup_{\tau \in [t,s]} \left[X_\tau - \inf_{[t,\tau]} X\right] \geq \veps\right)
\end{align*}
Both of these expressions are dominated by $\PP(R_1 \geq \veps)$. Indeed,
\be
\PP\!\left(\sup_{\tau \in [s,s']} \left[ \sup_{[s,\tau]} X - X_\tau \right] \geq \veps \right)\leq \PP\!\left(\sup_{s \in [0,1]} \left[ \sup_{[0,s]} X - X_s \right] \geq \veps \right) 
\ee 
and the supremum on the right hand side is dominated by $R_1$. Thus,
\begin{equation}
\PP\!\left(\sup_{\tau \in [s,s']} \left[ \sup_{[s,\tau]} X - X_\tau \right] \geq \veps \right) \leq \PP(R_1 \geq \veps) \;.
\label{eq:supsupinequality}
\end{equation}
Integrating the expression as a nested integral of $\PP(N^\veps \geq k)$, the variable $s_{k-1}$ between $t_{k-1}$ and $1$
\begin{align*}
\PP(N^\veps \geq k) &= \int_{\Sigma_{2k-3}} \PP(T_1^\veps = t_1)\PP(S_1 =s_1 \, \vert \, T_1 = t_1) \cdots \PP(t_{k-1}\leq S_{k-1}\leq 1\,\vert\, T_{k-1}=t_{k-1}) \; dt_{k-1}\prod_{i=1}^{k-2}ds_i dt_i  \\
&\leq  \PP(R \geq \veps) \int_{\Sigma_{2k-3}}\PP(T_1^\veps = t_1)\PP(S_1 =s_1 \, \vert \, T_1 = t_1) \cdots \PP(T_{k-1} = t_{k-1} \,\vert \, S_{k-2} = s_{k-2}) dt_{k-1}\prod_{i=1}^{k-2}ds_i dt_i \nonum
\end{align*}
Carrying out the subsequent $2k-3$ integrations and by repeated use of the inequality given in equation \ref{eq:supsupinequality}, we obtain the result
\be
\PP(N^\veps \geq k) \leq \PP(R \geq \veps)^{2k-2} \;.
\ee
By the hypothesis of the theorem, for all $\veps \geq \veps^*$, $\PP(R \geq \veps) <1$ so the above condition guarantees the summability (and absolute and uniform convergence) of the series $\expect{e^{\lambda N^\veps}}$ on the half plane $\Real(\lambda)> 2 \log(\PP(R_1 \geq \veps))$. Moreover the same summability condition holds for $N^{x,x+\veps}$, since the latter is dominated by $N^\veps$.
\end{proof}

\subsection{\textit{A priori} estimates on the asymptotic behaviour of small and large bars}
If $f$ is a continuous function, the asymptotic behaviour of $N^\veps$ is closely related to the regularity of $f$. For functions $f: [0,t] \to \R$, the correct notion of regularity to look at is the $p$-variation, of which we recall the definition.
\begin{definition}
Let $f: [0,t] \to \R$ be a function. The \textbf{(true) $p$-variation of $f$} is defined as
\be
\norm{f}_{p\text{-var}} := \sup_{\mathcal{P}} \left[\sum_{t_k \in \mathcal{P}} \abs{f(t_k)-f(t_{k-1})}^p \right]^{1/p}
\ee
where the supremum ranges over all finite partitions of $[0,t]$.
\end{definition}
The $p$-variation can be used to infer something about the asymptotic behaviour of $N^\veps_f$, as shown by the following theorem.
\begin{theorem}[Picard, \S 3\cite{Picard:Trees}]
\label{thm:pvarandupboxdim}
Given a continuous function $f: [0,1] \to \R$, 
\be
\mathcal{V}(f) = \updim T_f = \limsup_{\veps \to 0} \frac{\log N^\veps}{\log(1/\veps)} \vee 1
\ee
where $\updim$ denotes the upper-box dimension, $a \vee b = \max\{a,b\}$,
\be
\mathcal{V}(f) := \inf\{p \geq 1\, \vert \, \norm{f}_{\pvar} <\infty\} \,.
\ee
\end{theorem}
\begin{remark}
A general version of this theorem, applicable on more general metric spaces exists \cite{Perez_2020}.  
\end{remark}
\begin{corollary}
For any deterministic function $f: [0,1] \to \R$ and for every $\delta >0$, 
\be
N^\veps_f = O(\veps^{-\mathcal{V}(f)-\delta}) \quad \text{as }\veps \to 0 \,.
\ee
\end{corollary}
Having characterized the behavious of $N^\veps$ for small $\veps$, we can characterize the behaviour of large bars (in expectation) using theorem \ref{thm:MarkovNvepsRange}. 
\begin{corollary}
\label{cor:vepstoinfty}
Let $X$ be a non-constant stochastic process on $[0,t]$ defined on the filtered probability space $(\Omega, \mathcal{F},\PP)$ allowing an almost surely continuous modification and satisfying the strong Markov property. Suppose also that there exists some $\veps^*$ such that for all $\veps > \veps^*$, $\PP(R_t \geq \veps^*) <1$, then,
\be
\expect{N^\veps} \sim \PP(R_t \geq \veps) \quad \text{as } \veps \to \infty \,. 
\ee
\end{corollary}
\begin{proof}
Note $p_\veps =\PP(R_t \geq \veps)$. From the theorem, we deduce, 
\be
p_\veps \leq \expect{N^\veps} \leq p_\veps + \frac{p_\veps^2}{1-p_\veps^2} \quad \implies \quad \expect{N^\veps} \sim p_\veps \text{ as } \veps \to \infty \,.
\ee
\end{proof}

\section{Continuous semimartingales, local times and asymptotic behaviour of barcodes}
In the previous section, we quantified the asymptotics of $N^{x,x+\veps}$ an $N^\veps$ solely based on the regularity of the functions considered. For stochastic processes on $[0,t]$, we can refine this analysis by focusing on (continuous) semimartingales. Semimartingales constitute the largest class of processes with respect to which the Itô and Stratonovich integrals can be defined. In other words, they are a class of processes rich enough to be worthy of particular attention. For a comprehensive introduction to these objects and the probabilistic concepts included in this section, we kindly refer the reader to classical references on stochastic calculus \cite{LeGall:BrownianMotion,Revuz_1999}. 
\begin{definition}[Local martingales and semimartingales]
Let $(\Omega, \mathcal{F},\PP)$ be a filtered probability space and let $\mathcal{F}_* := (\mathcal{F}_t)_{t\geq 0}$ be the filtration of $\mathcal{F}$. An $\mathcal{F}_*$-adapted process $M : [0, \infty] \times \Omega \to \R$ is an \textbf{$\mathcal{F}_*$-local martingale} if there exists a sequence of stopping times $\tau_k$ such that
\begin{enumerate}
  \item The $\tau_k$ are a.s. increasing, \textit{i.e.} $\PP(\tau_k < \tau_{k+1}) =1$;
  \item The $\tau_k$ are a.s. divergent, \textit{i.e.} $\PP(\lim_{k\to \infty} \tau_k = \infty) =1$;
  \item The process $M_{t \wedge \tau_k}$ is an $\mathcal{F}_*$-martingale, \textit{i.e.} for all $s<t$, 
  \be
  \expect{M_{t \wedge \tau_k} \, \vert \, \mathcal{F}_s} = M_{s \wedge \tau_k} \,.
  \ee
\end{enumerate}
A process $(X_t)_{t \geq 0}$ is a \textbf{continuous semimartingale} if it can be written in the form
\be
X_t = M_t + A_t
\ee
where $A_t$ is a process of finite variation and $M_t$ is a continuous local martingale. 
\end{definition}

Throughout this section, we will use two key concepts stemming from the theory of stochastic integration: the local time of a (continuous) semimartingale and the quadratic variation. The latter is the easiest to define, given our introduction of the \textbf{true} quadratic variation in the previous section. 

\subsection{Quadratic variation and \textit{a priori} bounds}
\begin{definition}[Quadratic variation, Theorem 4.9 \cite{LeGall:BrownianMotion}]
Suppose that $(X_t)_t$ is a $\R$-valued stochastic process indexed by $\R_+$. The \textbf{quadratic variation of $X$} is the process defined as
\be
[X]_t := \lim_{\norm{\mathcal{P}} \to 0 } \sum_{k=1}^n (X_{t_k}-X_{t_{k-1}})^2 \,,
\ee
where $\mathcal{P}$ ranges over the set of finite partitions of the interval $[0,t]$ and $\norm{\mathcal{P}}$ the mesh of the partition $\mathcal{P}$. If it exists, the limit is taken in the sense of convergence in probability. Whenever $X$ is a continuous semimartingale, the quadratic variation always exists. 
\end{definition} 
\begin{remark}
This definition is strictly weaker than the \textbf{true} $2$-variation, which is defined as a supremum over the set of all finite partitions. In fact, for Brownian motion $B$, it is possible to show that $[B]_t = t$, but on the interval $[0,t]$, $\norm{B}_{2-\text{var}} = \infty$ almost surely.
\end{remark}
Understanding the quadratic variation is in some sense enough to understand local martingales, as the following theorem shows.
\begin{theorem}[Dambis-Dubins-Schwarz]
\label{thm:DDS}
Let $M$ be a continuous local martingale vanishing at $0$ such that $[M]_\infty = \infty$, then there exists a Brownian motion $B$ such that, a.s. for all $t\geq 0$, 
\be
M_t = B_{[M]_t} \,.
\ee 
\end{theorem}
This theorem allows us to give \textit{a priori} bounds on the small bar asymptotics of the barcode of a semimartingale. Indeed, since Brownian motion is a.s. $(\half-\delta)$-H\"older continuous for every $\delta >0$, the $p$-variation of any semimartingale is almost surely finite as soon as $p>2$ (since $p$-variation does not depend on parametrization). In particular we expect 
\be
N_X^\veps = O(\veps^{-2-\delta} ) \quad \text{as } \veps \to 0 \,.
\ee
for every $\delta >0$, by virtue of Picard's theorem. We may further refine this result by introducing the local time.

\subsection{The local time and sharp asymptotics}
The local time of a continuous semimartingale on an interval $[0,t]$ can be informally understood as the ``time spent'' by the process $X$ around a level $x \in \R$, in an equation
\be
L^x_X(t) = \int_0^t \delta(x-X_s) \; d[X]_s \,.
\ee
The above equation is informal and ill-defined, but gives an intuitive insight on what the local time represents and how it behaves. Formally, we define the local time as follows.
\begin{defprop}[Proposition 9.2 \cite{LeGall:BrownianMotion}]
Let $X$ be a continuous semimartingale and $x \in \R$. There exists an increasing process $(L_X^x(t))_{t \geq 0}$ such that the three following identities hold
\begin{align*}
\abs{X_t - x} &= \abs{X_0 - x} + \int_0^t \sgn(X_s-x) \; dX_s + L^x_X(t) \\
(X_t-x)_{+} &= (X_0 - x)_+ + \int_0^t \mathbf{1}_{X_s>x} \; dX_s + \half L_X^x(t) \\
(X_t-x)_{-} &= (X_0 - x)_- - \int_0^t \mathbf{1}_{X_s \leq x} \; dX_s + \half L_X^x(t) 
\end{align*}
The increasing process $(L^x_X(t))_{t \geq 0}$ is called the \textbf{local time of $X$ at level $x$}. Furthermore, for every stopping time $T$, the local time at $x$ of the stopped process $X$ at $T$, $L_{X^T}^x(t) = L_{X}^x(t \wedge T)$.
\end{defprop}
The local time of a process is useful, as it allows us to exchange time and space in integrations. A first useful result in this direction, which we will later use, is the following proposition. 
\begin{proposition}[Density occupation formula, Corollary 9.7 \cite{LeGall:BrownianMotion}]
Almost surely, for every non-negative, measurable function $\phi$ on $\R$,
\be
\int_0^t \phi(X_s) \; d[X]_s = \int_{\R} \phi(a) L^a_X(t) \;da \,.
\ee
\end{proposition}
This result can be informally derived using the informal definition of the local time we gave earlier. For a rigorous proof, we refer the reader to the cited reference. From the informal description we gave about the local time as the time spent by the process around a level $x$, we could expect the local time to be related to $N^{x,x+\veps}_X$ for small $\veps$. This turns out to be a well-known fact. 
\begin{proposition}[Approximation of the local time by downcrossings, \S VI, Theorem 1.10 \cite{Revuz_1999}]
\label{prop:DowncrossingLocalTime}
For every $t \geq 0$ and $s \geq 1$, if $X=M+A$ is a continuous semimartingale on $[0,t]$ and suppose that for $s \geq 1$ 
\be
\expect{[M]_t^{s/2} + \left(\int_0^t \abs{dA}_s\right)^s} < \infty \,.
\ee Then,
\be
\veps D_X^{x,x+\veps} \xrightarrow[\veps \to 0]{L^s} \half L_X^x(t) \,.
\ee
Moreover, in $L^s(\Omega)$,
\be
D_X^{x,x+\veps} \sim \frac{L^x_X(t))}{2\veps} + O(1) \quad \text{as } \veps \to 0 \,.
\ee
\end{proposition}
\begin{remark}
An analogous statement can be proven for upcrossings.
\end{remark}
\begin{theorem}
\label{thm:semimartingales}
Let $X=M+A$ be a continuous semimartingale on $[0,t]$ and suppose that for $s \geq 1$ 
\be
\expect{[M]_t^{s/2} + \left(\int_0^t \abs{dA}_s\right)^s} < \infty \,.
\ee 
Then in $L^s(\Omega)$, 
\begin{align*}
N_X^{x,x+\veps} &\sim \frac{L^x_X(t)}{2\veps} + O(1) \quad \text{as } \veps \to 0 \, \\
N_X^\veps &\sim \frac{[X]_t}{2\veps^2} + O(\veps^{-1}) \, \quad \text{as } \veps \to 0 \,.
\end{align*}
\end{theorem}
\begin{proof}
Since $0\leq N^{x,x+\veps}_X - D_X^{x,x+\veps} \leq 1$, the statement of proposition \ref{prop:DowncrossingLocalTime} is applicable to $N^{x,x+\veps}$ yielding the first result. It follows that in $L^s(\Omega)$,
\be
\abs{2\veps \lambda(T_X^\veps) - \int_{\R} L^x_X(t) \;dx} = O(\veps) \,.
\ee
Note that while the integral carries over $\R$, its support is contained within the compact $[\inf X, \sup X]$. We now use the density occupation formula for the local time where $\phi =1$, so 
\be
2\veps \lambda(T_X^\veps) = [X]_t  + O(\veps) \quad \text{ as } \veps \to 0 \text{ in } L^s(\Omega)
\ee 
We now use the fact that
\be
\lambda(T_X^\veps) = \int_\veps^\infty N^a \; da\,.
\ee
By monotonicity of $N^\veps$, for every $\delta >0$ small enough, almost surely, 
\be
N^{\veps(1+\delta)} \leq \frac{\lambda(T_X^\veps)- \lambda(T_X^{\veps(1+\delta)})}{\delta \veps} \leq \frac{1}{\delta \veps} \int_{\veps}^{\veps(1+\delta)} N^a \;da \leq N^\veps \,.
\ee
It follows that in $L^s(\Omega)$, for every $\delta>0$ small enough,
\be
\frac{[X]_t}{1+\delta} + O(\veps) \leq 2\veps^2 N^\veps \leq  \frac{[X]_t}{1-\delta} +O(\veps) \quad \text{ as } \veps \to 0 \,.
\ee
Taking $\delta \to 0$, we obtain the desired result.
\end{proof}

\section{Limiting processes}
An interesting point of view relative to these stability results and the results of this paper is to consider these irregular processes as almost sure $C^0$-limits of smooth processes, which have traditionally been more difficult to study. In this way, we can make affirmations about the barcodes of smooth processes up to some (small) error. This way of thinking is inspired by the study of trees, ultralimits and asymptotic cones in geometric group theory \cite{Roe:CoarseGeo}. 
\begin{proposition}
\label{prop:btsnbtmatching}
Let $(M,d)$ be a compact Polish metric space and let $X$ be an almost surely continuous stochastic process on $M$, defined on a probability space $(\Omega, \mathcal{F},\PP)$. Let $(X_n)_{n\in \N}$ be any sequence of continuous stochastic processes defined on the same probability space and suppose
\be
\delta_n := \norm{X-X_n}_{L^\infty(\Omega, L^\infty(M,\R))} \xrightarrow[n\to \infty]{} 0 \,.
\ee 
If for all $\veps >0$, $\expect{N^\veps_X}<\infty$ and is continuous in $\veps$, then for any $\veps \geq 2\delta_n$,
\be
N_{X_n}^{\veps} \xrightarrow[n\to \infty]{L^1} N_X^{\veps} \,, \nonum
\ee
Morever,
\be
\expect{\abs{N_X^\veps- N_{X_n}^\veps}} \leq \omega_\veps(\delta_n) \,, 
\ee
where $\omega_\veps(\delta) := \expect{N_X^{\veps-\delta}- N_X^{\veps+\delta}}$. With analogous hypotheses, the same statement holds for $N^{x,x+\veps}$.
\end{proposition}
\begin{proof}
The $L^\infty$-stability of barcodes with respect to the $\Linfty$-distance tells us that $\delta_n$ controls the bottleneck distance between the two barcodes \cite{Oudot:Persistence,Chazal:Persistence}. This implies that there exists a $\delta_n$-matching between the barcodes of $X_n$ and that of $X$, therefore:
\begin{itemize}
  \item If $\al \in \bcode(X)$ has length $\abs{\al} \geq 2\delta_n$, then $\exists! \,\beta \in \bcode(X_n)$ such that $(\al,\beta)$ are matched and the difference $\abs{\abs{\al} - \abs{\beta}} \leq 2\delta_n$;
  \item If $\beta \in \bcode(X_n)$ has length $\abs{\beta} \geq 2\delta_n$, then $\exists! \,\al \in \bcode(X)$ such that $(\al,\beta)$ are matched and the difference $\abs{\abs{\al} - \abs{\beta}} \leq 2\delta_n$;
  \item If $\al \in \bcode(X)$ or $\beta \in \bcode(X_n)$ in unmatched, then they have length $\leq 2\delta_n$.
\end{itemize}
It follows that for $\veps \geq 2\delta_n$ we have inequalities
\begin{align*}
N_{X_n}^{\veps+ 2\delta_n} &\leq N^{\veps}_X \\
N_{X}^{\veps + 2\delta_n} &\leq N^{\veps}_{X_n} \,,
\end{align*}
from which we obtain bounds on $N_{X_n}^\veps$
\be
N^{\veps+2\delta_n}_X \leq N^\veps_{X_n} \leq N^{\veps-2\delta_n}_X  \,.
\ee
These inequalities imply
\be
\abs{N_{X_n}^\veps -N_X^\veps} \leq \abs{N_X^{\veps+2\delta_n} - N_X^\veps} \vee \abs{N_X^{\veps-2\delta_n} - N_X^\veps} \nonum
\ee
Taking expectations of both sides, we have
\begin{align*}
\expect{\abs{N_{X_n}^\veps - N_X^\veps}} &\leq \expect{\abs{N_X^{\veps+2\delta_n} - N_X^\veps} \vee \abs{N_X^{\veps-2\delta_n} - N_X^\veps}} \nonum\\
&\leq \expect{\abs{N_X^{\veps+2\delta_n} - N_X^\veps}}+ \expect{\abs{N_X^{\veps-2\delta_n} - N_X^\veps}} \nonum \\
&= \expect{N_X^\veps - N_X^{\veps+2\delta_n}} + \expect{N_X^{\veps-2\delta_n}-N_X^\veps} \nonum \\
&\leq \omega_\veps(2\delta_n)
\end{align*}
by monotonicity of $N_X^\veps$. The right hand side of the inequality tends to $0$ as $n \to \infty$ by continuity of $\expect{N^\veps_X}$, so $N_{X_n}^\veps \xrightarrow[n\to \infty]{L^1} N_X^\veps$.
\end{proof}
\begin{remark}
The condition of uniform convergence over $\Omega$ can be quite restrictive, but covers some cases of distributions with compact supports in $C^0(X,\R)$. Moreover, it covers the case of processes derived from random point clouds $P$ (the filtration function is given by $d(-,P)$) stemming from a distribution with compact support on any Polish metric space $M$.
\end{remark}
We can adapt the proof of the above proposition to a setting relaxing the $L^\infty(\Omega,L^\infty(M,\R))$ convergence condition.
\begin{proposition}
\label{prop:LpLinftyCvgBars}
Keeping the same notation as before, suppose there exists a $p \geq 1$ such that
\be
\delta_n := \norm{X-X_n}_{L^p(\Omega, L^\infty(M,\R))} \xrightarrow[n\to \infty]{} 0 \,.
\ee 
Then, with probability $\geq 1- \frac{1}{a^p}$, for every $\veps \geq 2a\delta_n$,
\be
\abs{N^{\veps}_{X_n}- N^{\veps}_{X}} \leq N_X^{\veps-2a\delta_n} - N_X^{\veps +2a\delta_n} \,.
\ee
If $\expect{N_X^\veps}$ is continuous in $\veps$,
\be
\expect{\abs{N_{X_n}^\veps - N_X^\veps} \, \Big \vert \, \norm{X-X_n}_{\infty} \leq a\delta_n} \leq \omega_\veps(2a\delta_n) \,.
\ee
Moreover, 
\be
\PP\!\left(\abs{N_{X_n}^\veps - N_X^\veps} \geq k \right) \leq \frac{\omega_\veps(2a \delta_n)}{k} + \frac{1}{a^p} \quad \text{and} \quad  \PP\!\left(\abs{N_{X_n}^\veps - N_X^\veps} \geq k \; , \; \norm{X-X_n}_{\infty} \leq a\delta_n\right) \leq \frac{\omega_\veps(2a \delta_n)}{k}\,.
\ee
\end{proposition}
\begin{proof}
The proof goes as the previous one for a few minor exceptions. To obtain the probable bound on $\norm{X-X_n}_{L^\infty(M,\R)}$, we apply the Markov inequality. The rest of the proof follows by noticing that the conditional expectation is a contractive projection on $L^1$ and that
\be
\PP(A) =\PP(A \cap B) + \PP(A \cap B^c) \leq \PP(A \cap B) + \PP(B^c) \,.
\ee
\end{proof}

\begin{remark}
By an application of the Borel-Cantelli lemma, if $X_n$ tends to $X$ in $L^p(\Omega, L^\infty(M,\R))$ for $p>1$ at a rate $O(r(n))$ (where $r$ is a function tending to $0$ at infinity), then almost surely $X_n \xrightarrow[n\to \infty]{L^\infty} X$ at a rate $O(r(n))$ as well.
\end{remark}

\section{Applications}
\label{sec:Applications}
\subsection{Brownian motion and local martingales with deterministic strictly increasing quadratic variation}
For Brownian motion, it is possible to compute our quantities of interest exactly. 
\begin{theorem}[Perez, Proposition 4.4 \cite{Perez_2021}]
\label{thm:BMbars}
For Brownian motion $B$ on $[0,t]$, $\expect{N^\veps_B}$ admits the following series representations which converge well for large and small $\veps$ respectively
\begin{align*}
\expect{N^\veps_B} &= 4 \sum_{k\geq 1} (2k-1)\erfc\!\left(\frac{(2k-1)\veps}{\sqrt{2t}}\right) - k\,\erfc\!\left(\frac{2k\veps}{\sqrt{2t}}\right) \\
&= \frac{t}{2\veps^2} +\frac{2}{3} + 2\sum_{k\geq 1} (2(-1)^k -1) \frac{e^{-\pi^2 k^2 t/2\veps^2}t}{\veps^2}  \left[1 + \frac{\veps^2}{\pi^2k^2t}\right]\,.
\end{align*}
\end{theorem}
\begin{figure}[h!]
  \centering
    \includegraphics[width=0.7\textwidth]{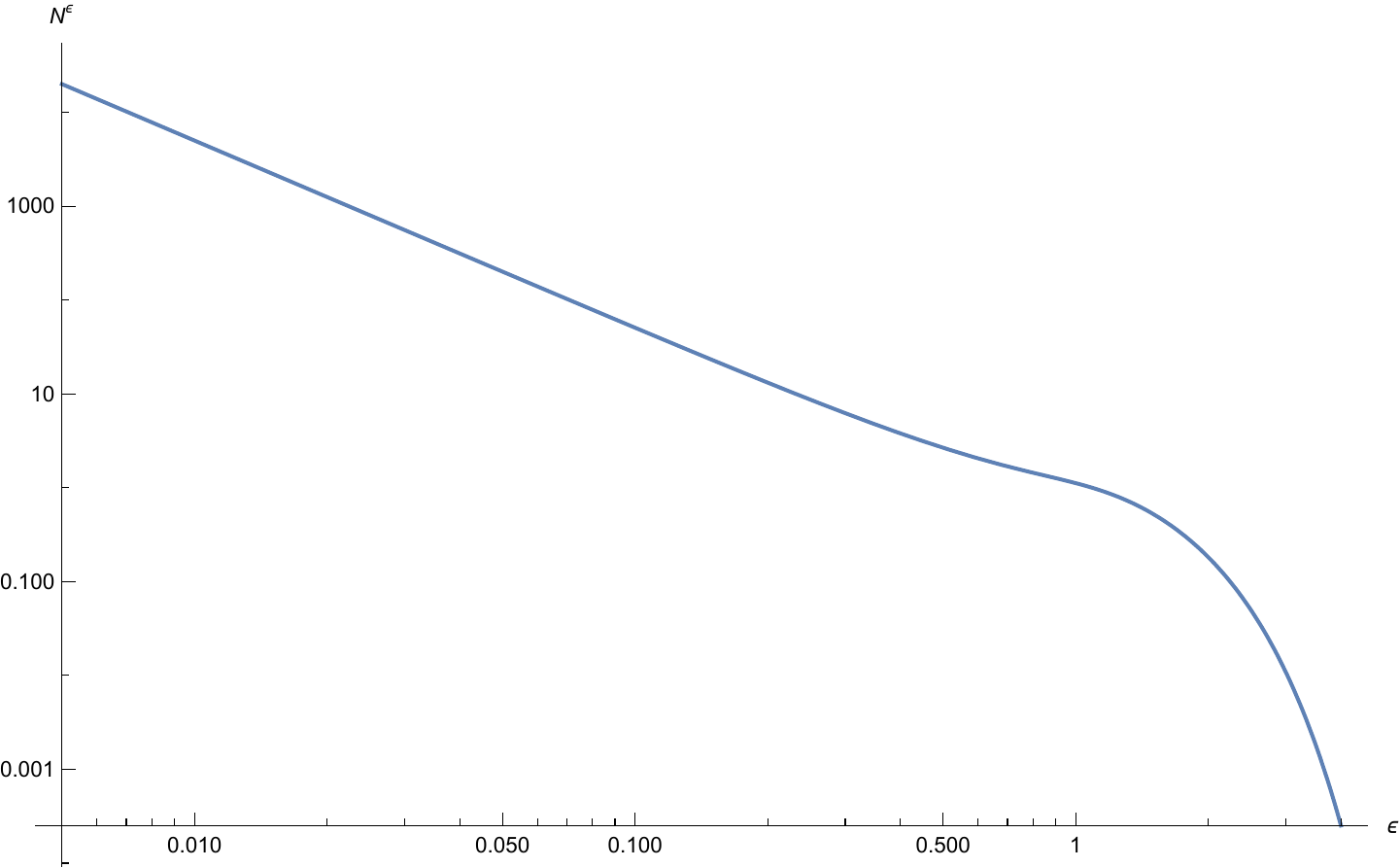}
  \caption{The expected number of bars of length $\geq \veps$, $\expect{N^\veps_B}$, as a function of $\veps$.}
\label{fig:BMCalculated}
\end{figure}
\begin{remark}
In particular, $\expect{N^\veps_B}$ is analytic in $\veps$ on an open wedge around the positive real axis in the complex plane, thereby confirming and extending Divol and Chazal's results. 
\end{remark}
\begin{proof}
We start by writing 
\begin{align*}
\expect{N_B^\veps} = \sum_{k=1}^\infty \PP(N^\veps \geq k) = \PP(R_t \geq \veps) + \sum_{k=2}^\infty \PP(S_{k-1}^\veps \leq t)
\end{align*}
Using the scale invariance of Brownian motion, 
\be
\expect{N_B^\veps} = \PP(R_1 \geq t^{-\half}\veps) + \sum_{k=2}^\infty \PP(S_{k-1}^1 \leq t/\veps^2)
\ee
We now notice that the stopping times 
\be
S_{k-1}^\veps = \sum_{i=1}^{k-1} (S_i^\veps - T_i^\veps) + (T_i^\veps - S_{i-1}^\veps) 
\ee
and that the increments $(S_i^\veps - T_i^\veps)$ and $(T_i^\veps - S_{i-1}^\veps)$ are independent and identically distributed. Moreover, in distribution, 
\be
\sup_{[0,t]} B - B_t = \abs{B_t} \,,
\ee
so that the $(S_i^\veps - T_i^\veps)$ and $(T_i^\veps - S_{i-1}^\veps)$ are distributed like the first hitting time of $\veps$ by $\abs{B_t}$. It is a classical result \cite[p.641]{Borodin_2002} that this hitting time $H^\veps$ satisfies
\be
\expect{e^{\lambda H^\veps}} = \sech(\veps\sqrt{2\lambda}) \,.
\ee
Similarly, it is also well-known that the range of Brownian motion satisfies
\be
\Lag_t(\PP(R_t \geq \veps))(\lambda) = \frac{\sech^2(\veps\sqrt{\frac{\lambda}{2}})}{\lambda}
\ee
We now take the Laplace transform with respect to the time variable of $\expect{N_B^\veps}$
\begin{align*}
\Lag_t(\expect{N_B^\veps})(\lambda) &= \frac{\sech^2(\veps\sqrt{\frac{\lambda}{2}})}{\lambda} + \frac{1}{\lambda} \sum_{k=2}^\infty \expect{e^{-\lambda \veps^2 S^1_{k-1}}} \\
&= \frac{\sech^2(\veps\sqrt{\frac{\lambda}{2}})}{\lambda} + \frac{1}{\lambda} \sum_{k=2}^\infty \expect{e^{-\lambda \veps^2 H^1}}^{2(k-1)} \,,
\end{align*}
where the last equality holds by virtue of i.i.d. character of the increments $(S_i^\veps - T_i^\veps)$ and $(T_i^\veps - S_{i-1}^\veps)$. Replacing the value of $\expect{e^{\lambda H^\veps}}$, 
\begin{align*}
\Lag_t(\expect{N_B^\veps})(\lambda) = \frac{\sech^2(\veps\sqrt{\frac{\lambda}{2}})+\csch^2(\veps \sqrt{2\lambda})}{\lambda} 
\end{align*}
The result is obtained by taking the inverse Laplace transform.
\end{proof}
Persistent homology is invariant under reparametrization. In particular if $\lambda: [0,t] \to [0,\lambda(t)]$ is an increasing bijection, $N_{f \circ \lambda}^\veps$ as calculated on $[0,\lambda(t)]$ is exactly equal to $N_{f}^\veps$ as calculated on $[0,t]$. Invoking the Dambis-Dubins-Schwarz theorem (theorem \ref{thm:DDS}), for every continuous local martingale $M$ (such that $M_0 = 0$) with deterministic and strictly increasing quadratic variation $[M]_t$, almost surely,
\be
N_M^\veps[0,t] =N_B^\veps[0,[M]_t]\,.
\ee
\begin{theorem}
\label{thm:locmart}
For any continuous local martingale $M$ on $[0,t]$ having deterministic and strictly increasing quadratic variation $[M]_t$ such that $[M]_\infty = \infty$, the same formula holds by replacing $t$ by $[M]_t$ in the result of theorem \ref{thm:BMbars}.
\end{theorem}
We can perform a similar calculation for $N^{x,x+\veps}$ when $x>0$, 
\begin{proposition}[Perez, Proposition 4.7 \cite{Perez_2021}]
\label{prop:BMNxxveps}
For Brownian motion on $[0,t]$ and $x>0$,
\begin{align*}
\expect{N^{x,x+\veps}_B} &= \sum_{k=1}^\infty \text{erfc}\left(\frac{x+(2 k-1)\veps}{\sqrt{2t}}\right) \\
&\sim \frac{1}{2\veps} \int_0^t \vp(x,s) \;ds +\sum_{k\geq 0} \frac{4 (-2)^{k}\!\left(2^{2k+1}-1\right) \zeta(2k+2)}{\pi^{2k+2}} \left[\frac{\del^{k}}{\del t^{k}} \vp(x,t)\right]\veps^{2k+1}  \; \text{as }\veps \to 0 \,,
\end{align*}
where $\vp(x,t)$ is the density of a centered Gaussian random variable of variance $t$ and $\zeta$ denotes the Riemann zeta function.
\end{proposition}
\begin{remark}
In the appropriate limits, we retrieve the results established for semimartingales. By virtue of the Dambis-Dubins-Schwarz theorem and by the invariance under reparametrization of persistent homology, we may also immediately extend the results and formul{\ae} of \cite{Perez_2021} to local martingales with deterministic and strictly increasing quadratic variation by replacing $t \mapsto [M]_t$.
\end{remark}

\subsection{Itô processes}
\begin{definition}
An \textbf{Itô process} is the solution to a stochastic differential equation of the form
\be
dX_t = \mu_t \, dt + \sigma_t \, dB_t
\ee
where $\mu_t$ and $\sigma_t$ are adapted processes.
\end{definition}
Every Itô process of this form has the strong Markov property and moreover, it is a semimartingale. We deduce that
\be
N_X^\veps \sim \frac{[X]_t}{2\veps^2} + O(\veps^{-1}) \text{ as } \veps \to 0 \,,
\ee
and an analogous expression for $N^{x,x+\veps}$ in terms of the local time of $X$. In expectation, we may always say something about $N^{x,x+\veps}_X$ by virtue of the following proposition.
\begin{proposition}
Let $X$ be an Itô process with deterministic quadratic variation, then
\be
\expect{L_X^x(t)} = \int_0^t \vp_X(x,s)\;d[X]_s \,,
\ee
where $\vp_X(-,s)$ is the density function of the random variable $X_s$ (which itself is a solution of the Fokker-Planck PDE associated to the SDE). 
\end{proposition}
\begin{proof}
We take the expectation of both sides of the occupation density formula with $\phi(a) = e^{-i\lambda a}$. 
\be
\int_0^t \expect{e^{-i\lambda X_s}} d[X]_s = \int_\R e^{-i\lambda a} \, \expect{L^a_X(t)} \;da 
\ee
The result follows from taking the inverse Fourier transform of both sides.
\end{proof}
\begin{proposition}[Asymptotics of barcodes of Itô processes]
Let $X$ be an Itô process with deterministic quadratic variation on $[0,t]$, then, 
\be
\expect{N^{x,x+\veps}} \sim \frac{1}{2\veps}\int_0^t \vp_X(x,s)\;d[X]_s + O(1) \quad \text{as } \veps \to 0 \,.
\ee
where $\vp_X(-,s)$ is the density function of the random variable $X_s$. 

\qed
\end{proposition}

As a Markov processes, Itô processes also satisfy
\be
\expect{N_X^\veps } \sim \PP(R_t \geq \veps) \quad \text{as } \veps \to \infty \,.
\ee
Finally, when $\mu_t = 0$ and $\sigma_t>0$ for all $t>0$ and is deterministic, an Itô process is a local martingale with quadratic variation
\be
[X]_t = \int_0^t \sigma_s^2 \;ds \,,
\ee
and, in particular, its barcode is then known completely.

\subsection{Limiting processes}
The Karhunen-Lo\`eve theorem \cite{Karhunen1947berLM,Loeve:ProbabilityII} shows that many well-known $C^0$-processes can be seen as limits of smooth processes and that the logic of studying smooth objects with their more irregular $C^0$ limits is a logic which can also be applicable in higher dimensions. Brownian motion itself can be seen as such a limit \cite{LeGall:BrownianMotion,BMPeresMorters}. Indeed, Paul L\'evy showed that if $(\xi_k)_{k \in \N}$ is a sequence of i.i.d. standard normal variables, the series
\be
\xi_0 t + \frac{\sqrt{2}}{\pi}\sum_{k=1}^\infty \xi_k \frac{\sin(\pi k t)}{k}
\ee
almost surely uniformly converges to the standard Brownian motion on $[0,1]$. Noting 
\be
S_n B := \xi_0 t+ \frac{\sqrt{2}}{\pi}\sum_{k=1}^{n-1} \xi_k \frac{\sin(\pi k t)}{k} \,,
\ee
it can be shown (for instance using \cite[Chapter 15, Theorem 4]{Kahane}) that for $p \geq 1$ there exists a finite constant $C_p$ such that
\be
\norm{B-S_n B}_{L^p(\Omega,L^{\infty}([0,1],\R))} \leq C_p n^{-\half} \log^\half(n) \,.
\ee
Applying the results of proposition \ref{prop:LpLinftyCvgBars} and optimizing in $a$, we have 
\be
\PP\!\left(\abs{N_{S_n B}^\veps - N_B^\veps} \geq k \right) = O\!\left(\left[ \frac{C_p \,n^{-\half} \log^\half(n)}{p\,\veps^3 k}\right]^{\frac{p}{p+1}}\right) \quad \text{as } \veps \to 0.
\ee
In particular, we know that through this approximation yields a curve on the $\log$-chart which doesn't stray far away from the Brownian motion's for $\veps \gtrsim (n^{-1}\log(n))^{\frac{1}{2}}$.
Brownian motion can also be approximated by random walks. The Komlós–Major–Tusnády (KMT) theorem provides a sharp estimate of the rate of convergence of these empirical processes to the Brownian bridge (which we will denote $W_t$).
\begin{definition}
Let $(X_n)_{n \in \N^*}$ be a sequence of (reduced, centered) i.i.d random variables. The \textbf{empirical process} defined by $X$ is the process
\be
\al^X_n(t) :=\left[\frac{1}{n}\sum_{k=1}^n \mathbf{1}_{]-\infty,t]}(X_k)\right]-t  \,.
\ee
\end{definition}
\begin{theorem}[KMT Theorem, \cite{Komlos_1975}]
Let $(U_n)_{n \in \N^*}$ be an i.i.d. sequence of uniform random variables on $[0,1]$. Then, there exists a Brownian bridge $(W_t)_{1\geq t \geq 0}$ such that for all $n \in \N^*$ and all $x > 0$
\be
\PP\!\left(\norm{\al_n^U(t) - W_t }_{\Linfty} > n^{-\half}(C \log(n)+x)\right) \leq Le^{-\lambda x} 
\ee
for some universal positive constants $C$, $L$ and $\lambda$ which are explicitly known \cite{Bretagnolle_1989}. 
\end{theorem}
Since $W$ is a semimartingale with quadratic variation $[W]_t = t$ on $[0,1]$, as it can be obtained as the solution to the SDE 
\be
dW_t = \frac{-W_t}{1-t} \; dt + dB_t \,.
\ee
We deduce that $N^\veps_W \sim \frac{1}{2\veps^2}$ as $\veps \to 0$. The KMT theorem and the same reasoning behind proposition \ref{prop:LpLinftyCvgBars} imply
\be
\PP\!\left(\abs{N_{\al_n^U}^\veps - N_W^\veps} \geq k \right) = O\left(\frac{n^{-\half}\log(n)}{\veps^3k}\right)\quad  \text{as }\veps \to 0 \,.
\ee
In particular, this approximation yields the curve $N^\veps_{\al_n^U}$ on the $\log$-chart which doesn't stray far away from $N^\veps_{W}$'s for $\veps \gtrsim n^{-\half}\log(n)$.

\section{Acknowledgements}
The author would like to thank Pierre Pansu and Claude Viterbo, without whom this work would not have been possible. Many thanks are owed to Jean-Fran\c{c}ois Le Gall and Nicolas Curien, for numerous fruitful discussions relevant to this work as well to the reviewers of this paper for their helpful comments.

%
%

\bibliographystyle{abbrv}
\bibliography{PhDThesis}

\end{document}